\makeatletter

\documentclass[
    11pt,
    a4paper,
    oneside,
    openright,
    center,
    chapterbib,
    crosshair,
    fleqn,
    headcount,
    headline,
    indent,
    indentfirst=false,
    portrait,
    phonetic,
    oldernstyle,
    onecolumn,
    sfbold,
    upper,
]
{amsart}

\let\th@plain\relax

\PassOptionsToPackage{T2A}{fontenc} 
\PassOptionsToPackage{utf8}{inputenc} 
\PassOptionsToPackage{cyrpart}{cyrillic}
\PassOptionsToPackage{british,ngerman,russian}{babel}
\PassOptionsToPackage{
    english,
    ngerman,
    russian,
    capitalise,
}{cleveref}
\PassOptionsToPackage{
    margin=10pt,
    position=bottom,
    justification=centering,
    font=small,
    labelformat=simple,
    labelfont={sc},
    labelsep={period},
    textfont={it}
}{caption}
\PassOptionsToPackage{
    margin=10pt,
    position=bottom,
    justification=centering,
    font=small,
    labelformat=parens,
    labelfont={bf},
    labelsep={space},
    textfont={it}
}{subcaption}
\PassOptionsToPackage{framemethod=TikZ}{mdframed}
\PassOptionsToPackage{
    bookmarks=true,
    bookmarksopen=false,
    bookmarksopenlevel=0,
    bookmarkstype=toc,
    colorlinks=false,
    raiselinks=true,
    hyperfigures=true,
}{hyperref}
\PassOptionsToPackage{normalem}{ulem}
\PassOptionsToPackage{
    amsmath,
    thmmarks,
}{ntheorem}
\PassOptionsToPackage{table}{xcolor}
\PassOptionsToPackage{
    all,
    color,
    curve,
    frame,
    import,
    knot,
    line,
    movie,
    rotate,
    textures,
    tile,
    tips,
    web,
    xdvi,
}{xy}
\PassOptionsToPackage{
    reset,
    left=1in,
    right=1in,
    top=20mm,
    bottom=20mm,
    heightrounded,
}{geometry}
\PassOptionsToPackage{
    symbol*, 
    multiple, 
}{footmisc}
\PassOptionsToPackage{overload}{textcase}
\PassOptionsToPackage{indentafter}{titlesec}

\usepackage{amsfonts}
\usepackage{amsmath}
\usepackage{amssymb}
\usepackage{ntheorem} 
\usepackage{array}
\usepackage{babel}
\usepackage{bbding}
\usepackage{bbm}
\usepackage{bibentry}
\usepackage{booktabs}
\usepackage{bold-extra} 
\usepackage{calc}
\usepackage{cancel}
\usepackage{caption} 
\usepackage{changepage}
\usepackage{cjhebrew}
\usepackage{cmlgc}
\usepackage{colonequals}
\usepackage{color}
\usepackage{comment}
\usepackage{datetime}
\usepackage{dsfont}
\usepackage{etex}
\usepackage{etoolbox}
\usepackage{eurosym}
\usepackage{fancybox}
\usepackage{fancyhdr}
\usepackage{float}
\usepackage{fontenc}
\usepackage{footmisc}
\usepackage{fp}
\usepackage{geometry}
\usepackage{graphicx}
\usepackage{ifpdf}
\usepackage{ifthen}
\usepackage{ifoddpage}
\usepackage{ifnextok}  
\usepackage{index}     
\usepackage{inputenc}
\usepackage{latexsym}
\usepackage{lineno}
\usepackage{listings}
\usepackage{longtable}
\usepackage{lscape}
\usepackage{mathrsfs}
\usepackage{multicol}
\usepackage{multirow}
\usepackage{nameref}
\usepackage{nowtoaux}
\usepackage{paralist}
\usepackage{enumerate} 
\usepackage{pgf}
\usepackage{pgfplots}
\usepackage{phonetic}
\usepackage{proof}
\usepackage{qtree}
\usepackage{refcount}
\usepackage{savesym}
\usepackage{stmaryrd}
\usepackage{synttree}
\usepackage{subcaption}
\usepackage{suffix}
\usepackage{yfonts} 
\usepackage{textcase} 
\usepackage{tikz}
\usepackage{xy}
\usepackage{wrapfig}
\usepackage{xcolor}
\usepackage{xspace}
\usepackage{xstring}
\usepackage{hyperref}
\usepackage{arydshln}
\usepackage{cleveref} 

\pgfplotsset{compat=newest}
\usetikzlibrary{math}

\usetikzlibrary{
    angles,
    arrows,
    automata,
    calc,
    decorations,
    decorations.pathmorphing,
    decorations.pathreplacing,
    positioning,
    patterns,
    quotes,
}

\savesymbol{Diamond}
\savesymbol{emptyset}
\savesymbol{ggg}
\savesymbol{int}
\savesymbol{lll}
\savesymbol{RectangleBold}
\savesymbol{langle}
\savesymbol{rangle}
\savesymbol{hookrightarrow}
\savesymbol{hookleftarrow}
\savesymbol{Asterisk}
\usepackage{mathabx}
\usepackage{wasysym}

\restoresymbol{x}{Diamond}
\restoresymbol{x}{emptyset}
\restoresymbol{x}{ggg}
\restoresymbol{x}{int}
\restoresymbol{x}{lll}
\restoresymbol{x}{RectangleBold}
\restoresymbol{x}{langle}
\restoresymbol{x}{rangle}
\restoresymbol{x}{hookrightarrow}
\restoresymbol{x}{hookleftarrow}
\restoresymbol{x}{Asterisk}

\ifpdf
    \usepackage{pdfcolmk}
\fi

\usepackage{mdframed}

\DeclareFontFamily{U}{MnSymbolA}{}
\DeclareFontShape{U}{MnSymbolA}{m}{n}{
    <-6> MnSymbolA5
    <6-7> MnSymbolA6
    <7-8> MnSymbolA7
    <8-9> MnSymbolA8
    <9-10> MnSymbolA9
    <10-12> MnSymbolA10
    <12-> MnSymbolA12
}{}
\DeclareFontShape{U}{MnSymbolA}{b}{n}{
    <-6> MnSymbolA-Bold5
    <6-7> MnSymbolA-Bold6
    <7-8> MnSymbolA-Bold7
    <8-9> MnSymbolA-Bold8
    <9-10> MnSymbolA-Bold9
    <10-12> MnSymbolA-Bold10
    <12-> MnSymbolA-Bold12
}{}
\DeclareSymbolFont{MnSyA}{U}{MnSymbolA}{m}{n}
\DeclareMathSymbol{\lcirclearrowright}{\mathrel}{MnSyA}{252}
\DeclareMathSymbol{\lcirclearrowdown}{\mathrel}{MnSyA}{255}
\DeclareMathSymbol{\rcirclearrowleft}{\mathrel}{MnSyA}{250}
\DeclareMathSymbol{\rcirclearrowdown}{\mathrel}{MnSyA}{251}

\DeclareFontFamily{U}{MnSymbolC}{}
\DeclareSymbolFont{MnSyC}{U}{MnSymbolC}{m}{n}
\DeclareFontShape{U}{MnSymbolC}{m}{n}{
    <-6>  MnSymbolC5
    <6-7>  MnSymbolC6
    <7-8>  MnSymbolC7
    <8-9>  MnSymbolC8
    <9-10> MnSymbolC9
    <10-12> MnSymbolC10
    <12->   MnSymbolC12%
}{}
\DeclareMathSymbol{\powerset}{\mathord}{MnSyC}{180}

\DeclareMathAlphabet{\mathpzc}{OT1}{pzc}{m}{it}

\def\boolwahr{true}
\def\boolfalsch{false}
\def\boolleer{}

\let\boolinappendix\boolfalsch
\let\boolinmdframed\boolfalsch
\let\eqtagset\boolfalsch
\let\eqtaglabel\boolleer
\let\eqtagsymb\boolleer

\newcount\bufferctr
\newcount\bufferreplace

\newlength\rtab
\newlength\gesamtlinkerRand
\newlength\gesamtrechterRand
\newlength\ownspaceabovethm
\newlength\ownspacebelowthm
\def\secnumberingpt{.}
\def\secnumberingseppt{.}
\def\subsecnumberingseppt{}
\def\thmnumberingpt{.}
\def\thmnumberingseppt{}
\def\thmForceSepPt{.}

\definecolor{leer}{gray}{1}
\definecolor{boxgrau}{gray}{0.85}
\definecolor{dunkelgrau}{gray}{0.5}
\definecolor{maroon}{rgb}{0.6901961,0.1882353,0.3764706}
\definecolor{dunkelgruen}{rgb}{0.015625,0.363281,0.109375}
\definecolor{dunkelrot}{rgb}{0.5450980392,0,0}
\definecolor{dunkelblau}{rgb}{0,0,0.5450980392}
\definecolor{blau}{rgb}{0,0,1}
\definecolor{newresult}{rgb}{0.6,0.6,0.6}
\definecolor{improvedresult}{rgb}{0.9,0.9,0.9}
\definecolor{hervorheben}{rgb}{0,0.9,0.7}
\definecolor{starkesblau}{rgb}{0.1019607843,0.3176470588,0.8156862745}
\definecolor{achtung}{rgb}{1,0.5,0.5}
\definecolor{frage}{rgb}{0.5,1,0.5}
\definecolor{schreibweise}{rgb}{0,0.7,0.9}
\definecolor{axiom}{rgb}{0,0.3,0.3}

\def\let@name#1#2{
    \expandafter\let\csname #1\expandafter\endcsname\csname #2\endcsname\relax
}
\DeclareRobustCommand\crfamily{\fontfamily{ccr}\selectfont}
\DeclareTextFontCommand{\textcr}{\crfamily}

\def\ifthenelseleer#1#2#3{\ifthenelse{\equal{#1}{}}{#2}{#1#3}}
\def\bedingtesspaceexpand#1#2#3{\ifthenelseleer{\csname #1\endcsname}{#3}{#2#3}}

\def\nvraum{\@ifnextchar\bgroup{\nvraum@c}{\nvraum@bes}}
    \def\nvraum@c#1{\vspace*{-#1\baselineskip}}
    \def\nvraum@bes{\vspace*{-\baselineskip}}
\def\erlaubeplatz{\relax\ifmmode\else\@\xspace\fi}
\def\entferneplatz{\relax\ifmmode\else\expandafter\@gobble\fi}

\def\send@toaux#1{\@bsphack\protected@write\@auxout{}{\string#1}\@esphack}

\def\rlabel#1[#2]#3#4#5{#5\rlabel@aux{#1}[#2]{#3}{#4}{#5}}
    \def\rlabel@aux#1[#2]#3#4#5{%
        \send@toaux{\newlabel{#1}{{\@currentlabel}{\thepage}{{\unexpanded{#5}}}{#2.\csname the#2\endcsname}{}}}\relax%
    }

\def\tag@rawscheme#1#2[#3]#4#5{\@ifnextchar[{\tag@rawscheme@{#1}{#2}[#3]{#4}{#5}}{\tag@rawscheme@{#1}{#2}[#3]{#4}{#5}[*]}}
    \def\tag@rawscheme@#1#2[#3]#4#5[#6]{\@ifnextchar\bgroup{\tag@rawscheme@@{#1}{#2}[#3]{#4}{#5}[#6]}{\tag@rawscheme@@{#1}{#2}[#3]{#4}{#5}[#6]{}}}
    \def\tag@rawscheme@@#1#2[#3]#4#5[#6]#7{%
        \ifthenelse{\equal{#6}{*}}{%
            \ifthenelse{\equal{#7}{\boolleer}}{\refstepcounter{#3}#4\csname the#3\endcsname#5}{#4#7#5}%
        }{%
            \refstepcounter{#3}#4%
            \ifthenelse{\equal{#7}{\boolleer}}{\rlabel{#6}[#3]{#1}{#2}{\csname the#3\endcsname}}{\rlabel{#6}[#3]{#1}{#2}{#7}}%
            #5%
        }%
    }
\def\tag@scheme#1#2[#3]{\tag@rawscheme{#1}{#2}[#3]{\upshape(}{\upshape)}}

\def\eqtag@post#1{\makebox[0pt][r]{#1}}
\def\eqtag@pre{\tag@scheme{Eq}{Equation}[Xe]}
\def\eqtag{\@ifnextchar[{\eqtag@}{\eqtag@[*]}}
    \def\eqtag@[#1]{\@ifnextchar\bgroup{\eqtag@@[#1]}{\eqtag@@[#1]{}}}
    \def\eqtag@@[#1]#2{\eqtag@post{\eqtag@pre[#1]{#2}}}

\def\eqcref#1{\text{(\ref{#1})}}

\def\punktcref#1{\eqcref{it:#1:\beweislabel}}

\def\opfromto[#1]_#2^#3{\underset{#2}{\overset{#3}{#1}}}

\def\mathclap#1{#1}
\def\oberunterset#1{\@ifnextchar^{\oberunterset@oben{#1}}{\oberunterset@unten{#1}}}
    \def\oberunterset@oben#1^#2_#3{\underset{\mathclap{#3}}{\overset{\mathclap{#2}}{#1}}}
    \def\oberunterset@unten#1_#2^#3{\underset{\mathclap{#2}}{\overset{\mathclap{#3}}{#1}}}
    \def\breitunderbrace#1_#2{\underbrace{#1}_{\mathclap{#2}}}
    \def\breitoverbrace#1^#2{\overbrace{#1}^{\mathclap{#2}}}
    \def\breitunderbracket#1_#2{\underbracket{#1}_{\mathclap{#2}}}
    \def\breitoverbracket#1^#2{\overbracket{#1}^{\mathclap{#2}}}

\def\generatenestedsecnumbering#1#2#3{%
    \expandafter\gdef\csname thelong#3\endcsname{%
        \expandafter\csname the#2\endcsname%
        \secnumberingpt%
        \expandafter\csname #1\endcsname{#3}%
    }%
    \expandafter\gdef\csname theshort#3\endcsname{%
        \expandafter\csname #1\endcsname{#3}%
    }%
}
\def\generatenestedthmnumbering#1#2#3{%
    \expandafter\gdef\csname the#3\endcsname{%
        \expandafter\csname the#2\endcsname%
        \thmnumberingpt%
        \expandafter\csname #1\endcsname{#3}%
    }%
    \expandafter\gdef\csname theshort#3\endcsname{%
        \expandafter\csname #1\endcsname{#3}%
    }%
}

\def\+#1{\addtocounter{#1}{1}}
\def\setcounternach#1#2{\setcounter{#1}{#2}\addtocounter{#1}{-1}}

\def\forcepunkt#1{#1\IfEndWith{#1}{.}{}{.}}
\def\lateinabkuerzung#1#2{%
    \expandafter\gdef\csname #1\endcsname{\emph{#2}\@ifnextchar.{\entferneplatz}{\erlaubeplatz}}
}
\def\deutscheabkuerzung#1#2{%
    \expandafter\gdef\csname #1\endcsname{{#2}\@ifnextchar.{\entferneplatz}{\erlaubeplatz}}
}

\def\matrix#1{\left(\begin{array}{#1}}
    \def\endmatrix{\end{array}\right)}
\def\smatrix{\left(\begin{smallmatrix}}
    \def\endsmatrix{\end{smallmatrix}\right)}

\def\multiargrekursiverbefehl#1#2#3#4#5#6#7#8{%
    \expandafter\gdef\csname#1\endcsname #2##1#4{\csname #1@anfang\endcsname##1#3\egroup}
    \expandafter\def\csname #1@anfang\endcsname##1#3{#5##1\@ifnextchar\egroup{\csname #1@ende\endcsname}{#7\csname #1@mitte\endcsname}}
    \expandafter\def\csname #1@mitte\endcsname##1#3{#6##1\@ifnextchar\egroup{\csname #1@ende\endcsname}{#7\csname #1@mitte\endcsname}}
    \expandafter\def\csname #1@ende\endcsname##1{#8}
}
\multiargrekursiverbefehl{svektor}{[}{;}{]}{\begin{smatrix}}{}{\\}{\\\end{smatrix}}
\multiargrekursiverbefehl{vektor}{[}{;}{]}{\begin{matrix}{c}}{}{\\}{\\\end{matrix}}
\multiargrekursiverbefehl{vektorzeile}{}{,}{;}{}{&}{}{}
\multiargrekursiverbefehl{matlabmatrix}{[}{;}{]}{\begin{smatrix}\vektorzeile}{\vektorzeile}{;\\}{;\end{smatrix}}

\def\faelle[#1]#2{\left\{\begin{array}[#1]{#2}}
    \def\endfaelle{\end{array}\right.}

\def\BeweisRichtung[#1]{\@ifnextchar\bgroup{\@BeweisRichtung@c[#1]}{\@BeweisRichtung@bes[#1]}}
    \def\@BeweisRichtung@bes[#1]{{\bfseries(#1).~}}
    \def\@BeweisRichtung@c[#1]#2#3{{\bfseries(#2#1#3).~}}
\def\erzeugeBeweisRichtungBefehle#1#2{
    \expandafter\gdef\csname #1text\endcsname##1##2{\BeweisRichtung[#2]{##1}{##2}}
    \expandafter\gdef\csname #1\endcsname{%
        \@ifnextchar\bgroup{\csname #1@\endcsname}{\csname #1text\endcsname{}{}}%
    }
    \expandafter\gdef\csname #1@\endcsname##1##2{%
        \csname #1text\endcsname{\punktcref{##1}}{\punktcref{##2}}%
    }
}
\erzeugeBeweisRichtungBefehle{hinRichtung}{$\Longrightarrow$}
\erzeugeBeweisRichtungBefehle{herRichtung}{$\Longleftarrow$}
\erzeugeBeweisRichtungBefehle{hinherRichtung}{$\Longleftrightarrow$}

\def\cal#1{\mathcal{#1}}

\def\mathfrak#1{\mbox{\usefont{U}{euf}{m}{n}#1}}

\def\rectangleblack{\text{\RectangleBold}}

\def\squareblack{\blacksquare}

\def\crefname@full#1#2#3#4#5{%
    \crefname{#1}{#2}{#3}
    \Crefname{#1}{#4}{#5}
}
\def\crefname@fullmod#1#2#3#4#5{%
    \crefname@full{#1}{#2}{#3}{#4}{#5}
    \crefname@full{#1@basic}{#2}{#3}{#4}{#5}
    \crefname@full{#1@withName}{#2}{#3}{#4}{#5}
}
\crefname@full{chapter}{chapter}{chapters}{Chapter}{Chapters}
\crefname@full{appendix}{appendix}{appendices}{Appendix}{Appendices}
\crefname@full{section}{section}{sections}{Section}{Sections}
\crefname@full{subsection}{section}{sections}{Section}{Sections}
\crefname@full{subsubsection}{section}{sections}{Section}{Sections}
\crefname@full{subsubsubsection}{section}{sections}{Section}{Sections}
\crefname@full{table}{table}{tables}{Table}{Tables}
\crefname@full{figure}{figure}{figures}{Figure}{Figures}
\crefname@full{subfigure}{figure}{figures}{Figure}{Figures}

\crefname@fullmod{thm}{theorem}{theorems}{Theorem}{Theorems}
\crefname@fullmod{thmStar}{theorem}{theorems}{Theorem}{Theorems}
\crefname@fullmod{conj}{conjecture}{conjectures}{Conjecture}{Conjectures}
\crefname@fullmod{cor}{corollary}{corollaries}{Corollary}{Corollaries}
\crefname@fullmod{defn}{definition}{definitions}{Definition}{Definitions}
\crefname@fullmod{conv}{convention}{conventions}{Convention}{Conventions}
\crefname@fullmod{e.g.}{example}{examples}{Example}{Examples}
\crefname@fullmod{prop}{proposition}{propositions}{Proposition}{Propositions}
\crefname@fullmod{proof}{proof}{proofs}{Proof}{Proofs}
\crefname@fullmod{lemm}{lemma}{lemmata}{Lemma}{Lemmata}
\crefname@fullmod{qstn}{question}{questions}{Question}{Questions}
\crefname@fullmod{rem}{remark}{remarks}{Remark}{Remarks}

    \def\qedEIGEN#1{\@ifnextchar[{\qedEIGEN@c{#1}}{\qedEIGEN@bes{#1}}}
    \def\qedEIGEN@bes#1{%
        \parfillskip=0pt
        \widowpenalty=10000
        \displaywidowpenalty=10000
        \finalhyphendemerits=0
        \leavevmode
        \unskip
        \nobreak
        \hfil
        \penalty50
        \hskip.2em
        \null
        \hfill
        #1
        \par%
    }
    \def\qedEIGEN@c#1[#2]{%
        \parfillskip=0pt
        \widowpenalty=10000
        \displaywidowpenalty=10000
        \finalhyphendemerits=0
        \leavevmode
        \unskip
        \nobreak
        \hfil
        \penalty50
        \hskip.2em
        \null
        \hfill
        {#1~{\small\bfseries\upshape (#2)}}%
        \par%
    }
    \def\qedVARIANT#1#2{
        \expandafter\def\csname ennde#1Sign\endcsname{#2}
        \expandafter\def\csname ennde#1\endcsname{\@ifnextchar[{\qedEIGEN@c{#2}}{\qedEIGEN@bes{#2}}} 
    }
    \qedVARIANT{OfProof}{$\squareblack$}
    \qedVARIANT{OfWork}{\rectangleblack}
    \qedVARIANT{OfSomething}{$\dashv$}
    \qedVARIANT{OnNeutral}{} 

    \def\ra@pretheoremwork{
        \setlength{\theorempreskipamount}{\ownspaceabovethm}
    }
    \def\rathmtransfer#1#2{
        \expandafter\def\csname #2\endcsname{\csname #1\endcsname}
        \expandafter\def\csname end#2\endcsname{\csname end#1\endcsname}
    }

    \def\ranewthm#1#2#3[#4]{
        \theoremstyle{\current@theoremstyle}
        \theoremseparator{\current@theoremseparator}
        \theoremprework{\ra@pretheoremwork}
        \@ifundefined{#1@basic}{\newtheorem{#1@basic}[#4]{#2}}{\renewtheorem{#1@basic}[#4]{#2}}
        \theoremstyle{\current@theoremstyle}
        \theoremseparator{\thmForceSepPt}
        \theoremprework{\ra@pretheoremwork}
        \@ifundefined{#1@withName}{\newtheorem{#1@withName}[#4]{#2}}{\renewtheorem{#1@withName}[#4]{#2}}
        \theoremstyle{nonumberplain}
        \theoremseparator{\thmForceSepPt}
        \theoremprework{\ra@pretheoremwork}
        \@ifundefined{#1@star@basic}{\newtheorem{#1@star@basic}[#4]{#2}}{\renewtheorem{#1@star@basic}[#4]{#2}}
        \theoremstyle{nonumberplain}
        \theoremseparator{\thmForceSepPt}
        \theoremprework{\ra@pretheoremwork}
        \@ifundefined{#1@star@withName}{\newtheorem{#1@star@withName}[#4]{#2}}{\renewtheorem{#1@star@withName}[#4]{#2}}
        \umbauenenv{#1}{#3}[#4]
        \umbauenenv{#1@star}{#3}[#4]
        \rathmtransfer{#1@star}{#1*}
    }

    \def\umbauenenv#1#2[#3]{%
        \expandafter\def\csname #1\endcsname{\relax%
            \@ifnextchar[{\csname #1@\endcsname}{\csname #1@\endcsname[*]}%
        }
        \expandafter\def\csname #1@\endcsname[##1]{\relax%
            \@ifnextchar[{\csname #1@@\endcsname[##1]}{\csname #1@@\endcsname[##1][*]}%
        }
        \expandafter\def\csname #1@@\endcsname[##1][##2]{%
            \ifx*##1%
                \def\enndeOfBlock{\csname end#1@basic\endcsname}
                \csname #1@basic\endcsname%
            \else%
                \def\enndeOfBlock{\csname end#1@withName\endcsname}
                \csname #1@withName\endcsname[##1]%
            \fi%
            \def\makelabel####1{%
                \gdef\beweislabel{####1}%
                \label{\beweislabel}%
            }%
            \ifx*##2%
                \def\enndeSymbol{\qedEIGEN{#2}}
            \else%
                \def\enndeSymbol{\qedEIGEN{#2}[##2]}
            \fi
        }
        \expandafter\gdef\csname end#1\endcsname{\enndeSymbol\enndeOfBlock}
    }

        \def\current@theoremstyle{plain}
        \def\current@theoremseparator{\thmnumberingseppt}
        \theoremstyle{\current@theoremstyle}
        \theoremseparator{\current@theoremseparator}
        \theoremsymbol{}

    \generatenestedthmnumbering{arabic}{section}{X}
    \generatenestedthmnumbering{arabic}{section}{Xe}
    \generatenestedthmnumbering{Roman}{section}{Xsp}

        \theoremheaderfont{\upshape\bfseries}
        \theorembodyfont{\slshape}

    \ranewthm{thm}{Theorem}{\enndeOnNeutralSign}[X]
    \ranewthm{lemm}{Lemma}{\enndeOnNeutralSign}[X]
    \ranewthm{cor}{Corollary}{\enndeOnNeutralSign}[X]
    \ranewthm{prop}{Proposition}{\enndeOnNeutralSign}[X]

        \theorembodyfont{\upshape}

    \ranewthm{defn}{Definition}{\enndeOnNeutralSign}[X]
    \ranewthm{conv}{Convention}{\enndeOnNeutralSign}[X]
    \ranewthm{e.g.}{Example}{\enndeOnNeutralSign}[X]
    \ranewthm{fact}{Fact}{\enndeOnNeutralSign}[X]
    \ranewthm{rem}{Remark}{\enndeOnNeutralSign}[X]
    \ranewthm{qstn}{Question}{\enndeOnNeutralSign}[X]

        \theoremheaderfont{\itshape\bfseries}
        \theorembodyfont{\upshape}

    \ranewthm{proof@tmp}{Proof}{\enndeOfProofSign}[Xdisplaynone]
    \rathmtransfer{proof@tmp*}{proof}

    \def\behauptungbeleg@claim{%
        \iflanguage{british}{Claim}{%
        \iflanguage{english}{Claim}{%
        \iflanguage{ngerman}{Behauptung}{%
        \iflanguage{russian}{Утверждение}{%
        Claim%
        }}}}%
    }
    \def\behauptungbeleg@pf@kurz{%
        \iflanguage{british}{Pf}{%
        \iflanguage{english}{Pf}{%
        \iflanguage{ngerman}{Bew}{%
        \iflanguage{russian}{Доказательство}{%
        Pf%
        }}}}%
    }
    \def\behauptungbeleg{\@ifnextchar\bgroup{\behauptungbeleg@c}{\behauptungbeleg@bes}}
            \def\behauptungbeleg@c#1{\item[{\bfseries \behauptungbeleg@claim\erlaubeplatz #1.}]}
            \def\behauptungbeleg@bes{\item[{\bfseries \behauptungbeleg@claim.}]}
        \def\belegbehauptung{\item[{\bfseries\itshape\behauptungbeleg@pf@kurz.}]}

    \newdateformat{standardshort}{\oldstylenums{\THEYEAR}.\oldstylenums{\THEMONTH}.\oldstylenums{\THEDAY}}
    \newdateformat{standardcompact}{\THEYEAR\twodigit{\THEMONTH}\twodigit{\THEDAY}}
    \newdateformat{standardlong}{\THEYEAR\ \monthname\ \THEDAY}
    \newcolumntype{\RECHTS}[1]{>{\raggedleft}p{#1}}
    \newcolumntype{\LINKS}[1]{>{\raggedright}p{#1}}
    \newcolumntype{m}{>{$}l<{$}}
    \newcolumntype{C}{>{$}c<{$}}
    \newcolumntype{L}{>{$}l<{$}}
    \newcolumntype{R}{>{$}r<{$}}
    \newcolumntype{0}{@{\hspace{0pt}}}
    \newcolumntype{\LINKSRAND}{@{\hspace{\@totalleftmargin}}}
    \newcolumntype{h}{@{\extracolsep{\fill}}}
    \newcolumntype{i}{>{\itshape}}
    \newcolumntype{t}{@{\hspace{\tabcolsep}}}
    \newcolumntype{q}{@{\hspace{1em}}}
    \newcolumntype{n}{@{\hspace{-\tabcolsep}}}
    \newcolumntype{M}[2]{%
        >{\begin{minipage}{#2}\begin{math}}%
        {#1}%
        <{\end{math}\end{minipage}}%
    }
    \newcolumntype{T}[2]{%
        >{\begin{minipage}{#2}}%
        {#1}%
        <{\end{minipage}}%
    }
    \def\punkteumgebung@genbefehl#1#2#3{
        \punkteumgebung@genbefehl@{#1}{#2}{#3}{}{}
        \punkteumgebung@genbefehl@{multi#1}{#2}{#3}{
            \setlength{\columnsep}{10pt}%
            \setlength{\columnseprule}{0pt}%
            \begin{multicols}{\thecolumnanzahl}%
        }{\end{multicols}\nvraum{1}}
    }
    \def\punkteumgebung@genbefehl@#1#2#3#4#5{
        \expandafter\gdef\csname #1\endcsname{
            \@ifnextchar\bgroup{\csname #1@c\endcsname}{\csname #1@bes\endcsname}
        }
            \expandafter\def\csname #1@c\endcsname##1{
                \@ifnextchar[{\csname #1@c@\endcsname{##1}}{\csname #1@c@\endcsname{##1}[\z@]}
            }
            \expandafter\def\csname #1@c@\endcsname##1[##2]{
                \@ifnextchar[{\csname #1@c@@\endcsname{##1}[##2]}{\csname #1@c@@\endcsname{##1}[##2][\z@]}
            }
            \expandafter\def\csname #1@c@@\endcsname##1[##2][##3]{
                \let\alterlinkerRand\gesamtlinkerRand
                \let\alterrechterRand\gesamtrechterRand
                \addtolength{\gesamtlinkerRand}{##2}
                \addtolength{\gesamtrechterRand}{##3}
                \advance\linewidth -##2%
                \advance\linewidth -##3%
                \advance\@totalleftmargin ##2%
                \parshape\@ne \@totalleftmargin\linewidth%
                #4
                \begin{#2}[\upshape ##1]%
                    \setlength{\parskip}{0.5\baselineskip}\relax%
                    \setlength{\topsep}{\z@}\relax%
                    \setlength{\partopsep}{\z@}\relax%
                    \setlength{\parsep}{\parskip}\relax%
                    \setlength{\itemsep}{#3}\relax%
                    \setlength{\listparindent}{\z@}\relax%
                    \setlength{\itemindent}{\z@}\relax%
            }
            \expandafter\def\csname #1@bes\endcsname{
                \@ifnextchar[{\csname #1@bes@\endcsname}{\csname #1@bes@\endcsname[\z@]}
            }
            \expandafter\def\csname #1@bes@\endcsname[##1]{
                \@ifnextchar[{\csname #1@bes@@\endcsname[##1]}{\csname #1@bes@@\endcsname[##1][\z@]}
            }
            \expandafter\def\csname #1@bes@@\endcsname[##1][##2]{
                \let\alterlinkerRand\gesamtlinkerRand
                \let\alterrechterRand\gesamtrechterRand
                \addtolength{\gesamtlinkerRand}{##1}
                \addtolength{\gesamtrechterRand}{##2}
                \advance\linewidth -##1%
                \advance\linewidth -##2%
                \advance\@totalleftmargin ##1%
                \parshape\@ne \@totalleftmargin\linewidth%
                #4
                \begin{#2}%
                    \setlength{\parskip}{0.5\baselineskip}\relax%
                    \setlength{\topsep}{\z@}\relax%
                    \setlength{\partopsep}{\z@}\relax%
                    \setlength{\parsep}{\parskip}\relax%
                    \setlength{\itemsep}{#3}\relax%
                    \setlength{\listparindent}{\z@}\relax%
                    \setlength{\itemindent}{\z@}\relax%
            }
        \expandafter\gdef\csname end#1\endcsname{%
            \end{#2}#5
            \setlength{\gesamtlinkerRand}{\alterlinkerRand}
            \setlength{\gesamtlinkerRand}{\alterrechterRand}
        }
    }

    \def\ritempunkt{{\Large \textbullet}} 
    \setdefaultitem{\ritempunkt}{\ritempunkt}{\ritempunkt}{\ritempunkt}
    \punkteumgebung@genbefehl{itemise}{compactitem}{\parskip}{}{}
    \punkteumgebung@genbefehl{kompaktitem}{compactitem}{\z@}{}{}
    \punkteumgebung@genbefehl{enumerate}{compactenum}{\parskip}{}{}
    \punkteumgebung@genbefehl{kompaktenum}{compactenum}{\z@}{}{}

    \renewenvironment{thebibliography}[1]{%
        \begin{ALTthebibliography}{#1}
        \addcontentsline{toc}{part}{\bibname}
    }{%
        \end{ALTthebibliography}
    }

    \def\matrix#1{\left(\begin{array}[mc]{#1}}
        \def\endmatrix{\end{array}\right)}
    \def\smatrix{\left(\begin{smallmatrix}}
        \def\endsmatrix{\end{smallmatrix}\right)}

    \def\multiargrekursiverbefehl#1#2#3#4#5#6#7#8{%
        \expandafter\gdef\csname#1\endcsname #2##1#4{\csname #1@anfang\endcsname##1#3\egroup}
        \expandafter\def\csname #1@anfang\endcsname##1#3{#5##1\@ifnextchar\egroup{\csname #1@ende\endcsname}{#7\csname #1@mitte\endcsname}}
        \expandafter\def\csname #1@mitte\endcsname##1#3{#6##1\@ifnextchar\egroup{\csname #1@ende\endcsname}{#7\csname #1@mitte\endcsname}}
        \expandafter\def\csname #1@ende\endcsname##1{#8}
    }
    \multiargrekursiverbefehl{svektor}{[}{;}{]}{\begin{smatrix}}{}{\\}{\\\end{smatrix}}
    \multiargrekursiverbefehl{vektor}{[}{;}{]}{\begin{matrix}{c}}{}{\\}{\\\end{matrix}}
    \multiargrekursiverbefehl{vektorzeile}{}{,}{;}{}{&}{}{}
    \multiargrekursiverbefehl{matlabmatrix}{[}{;}{]}{\begin{smatrix}\vektorzeile}{\vektorzeile}{;\\}{;\end{smatrix}}

    \def\underbracenodisplay#1{%
        \mathop{\vtop{\m@th\ialign{##\crcr
        $\hfil\displaystyle{#1}\hfil$\crcr
        \noalign{\kern3\p@\nointerlineskip}%
        \upbracefill\crcr\noalign{\kern3\p@}}}}\limits%
    }

    \def\mathe[#1]#2{%
        \ifthenelse{\equal{\boolinmdframed}{\boolwahr}}{}{\begin{escapeeinzug}}
        \noindent%
        \let\eqtagset\boolfalsch
        \let\eqtaglabel\boolleer
        \let\eqtagsymb\boolleer
        \let\alteqtag\eqtag
        \def\eqtag{\@ifnextchar[{\eqtag@loc@}{\eqtag@loc@[*]}}%
        \def\eqtag@loc@[##1]{\@ifnextchar\bgroup{\eqtag@loc@@[##1]}{\eqtag@loc@@[##1]{}}}%
        \def\eqtag@loc@@[##1]##2{%
            \gdef\eqtagset{\boolwahr}
            \gdef\eqtaglabel{##1}
            \gdef\eqtagsymb{##2}
        }%
        \def\verticalalign{}%
            \IfBeginWith{#1}{t}{\def\verticalalign{t}}{}%
            \IfBeginWith{#1}{m}{\def\verticalalign{c}}{}%
            \IfBeginWith{#1}{b}{\def\verticalalign{b}}{}%
        \def\horizontalalign{\null\hfill\null}%
            \IfEndWith{#1}{l}{}{\null\hfill\null}%
            \IfEndWith{#1}{r}{\def\horizontalalign{}}{}%
        \begin{math}
        \begin{array}[\verticalalign]{0#2}%
    }
        \def\endmathe{%
            \end{array}
            \end{math}\horizontalalign%
            \let\eqtag\alteqtag
            \ifthenelse{\equal{\eqtagset}{\boolwahr}}{\eqtag[\eqtaglabel]{\eqtagsymb}}{}
            \ifthenelse{\equal{\boolinmdframed}{\boolwahr}}{}{\end{escapeeinzug}}%
        }

    \def\longmathe[#1]#2{\relax
        \let\altarraystretch\arraystretch
        \renewcommand\arraystretch{1.2}\relax
        \begin{longtable}[#1]{\LINKSRAND #2}
    }
        \def\endlongmathe{
            \end{longtable}
            \renewcommand\arraystretch{\altarraystretch}
        }

    \def\einzug{\@ifnextchar[{\indents@}{\indents@[\z@]}}
        \def\indents@[#1]{\@ifnextchar[{\indents@@[#1]}{\indents@@[#1][\z@]}}
        \def\indents@@[#1][#2]{%
            \begin{list}{}{\relax
                \setlength{\topsep}{\z@}\relax
                \setlength{\partopsep}{\z@}\relax
                \setlength{\parsep}{\parskip}\relax
                \setlength{\listparindent}{\z@}\relax
                \setlength{\itemindent}{\z@}\relax
                \setlength{\leftmargin}{#1}\relax
                \setlength{\rightmargin}{#2}\relax
                \let\alterlinkerRand\gesamtlinkerRand
                \let\alterrechterRand\gesamtrechterRand
                \addtolength{\gesamtlinkerRand}{#1}
                \addtolength{\gesamtrechterRand}{#2}
            }\relax
                \item[]\relax
        }
            \def\endeinzug{%
                \setlength{\gesamtlinkerRand}{\alterlinkerRand}
                \setlength{\gesamtlinkerRand}{\alterrechterRand}
                \end{list}%
            }

    \def\escapeeinzug{\begin{einzug}[-\gesamtlinkerRand][-\gesamtrechterRand]}
        \def\endescapeeinzug{\end{einzug}}

    \def\programmiercode{
        \modulolinenumbers[1]
        \begin{einzug}[\rtab][\rtab]%
        \begin{linenumbers}%
            \fontfamily{cmtt}\fontseries{m}\fontshape{u}\selectfont%
            \setlength{\parskip}{1\baselineskip}%
            \setlength{\parindent}{0pt}%
    }
        \def\endprogrammiercode{
            \end{linenumbers}
            \end{einzug}
        }

    \def\schattiertebox@genbefehl#1#2#3{
        \expandafter\gdef\csname #1\endcsname{%
            \@ifnextchar[{\csname #1@args\endcsname}{\csname #1@args\endcsname[#3]}
        }
            \expandafter\def\csname #1@args\endcsname[##1]{%
                \@ifnextchar[{\csname #1@args@l\endcsname[##1]}{\csname #1@args@n\endcsname[##1]}
            }
            \expandafter\def\csname #1@args@l\endcsname[##1][##2]{%
                \@ifnextchar[{\csname #1@args@l@r\endcsname[##1][##2]}{\csname #1@args@l@n\endcsname[##1][##2]}
            }
            \expandafter\def\csname #1@args@n\endcsname[##1]{%
                \let\boolinmdframed\boolwahr
                \begin{mdframed}[#2leftmargin=0,rightmargin=0,outermargin=0,innermargin=0,##1]
            }
            \expandafter\def\csname #1@args@l@n\endcsname[##1][##2]{%
                \let\boolinmdframed\boolwahr
                \begin{mdframed}[#2leftmargin=##2/2,rightmargin=##2/2,outermargin=##2/2,innermargin=##2/2,##1]
            }
            \expandafter\def\csname #1@args@l@r\endcsname[##1][##2][##3]{%
                \let\boolinmdframed\boolwahr
                \begin{mdframed}[#2leftmargin=##2,rightmargin=##3,outermargin=##2,innermargin=##3,##1]
            }
        \expandafter\gdef\csname end#1\endcsname{%
            \end{mdframed}
            \let\boolinmdframed\boolfalsch
        }
    }
        \schattiertebox@genbefehl{schattiertebox}{
            splittopskip=0,%
            splitbottomskip=0,%
            frametitleaboveskip=0,%
            frametitlebelowskip=0,%
            skipabove=1\baselineskip,%
            skipbelow=1\baselineskip,%
            linewidth=2pt,%
            linecolor=black,%
            roundcorner=4pt,%
        }{
            backgroundcolor=leer,%
            nobreak=true,%
        }

        \schattiertebox@genbefehl{schattierteboxdunn}{
            splittopskip=0,%
            splitbottomskip=0,%
            frametitleaboveskip=0,%
            frametitlebelowskip=0,%
            skipabove=1\baselineskip,%
            skipbelow=1\baselineskip,%
            linewidth=1pt,%
            linecolor=black,%
            roundcorner=2pt,%
        }{
            backgroundcolor=leer,%
            nobreak=true,%
        }

    \def\tikzsetzepfeil#1{%
        \begin{tikzpicture}[remember picture,overlay,>=latex]%
            \draw #1;%
        \end{tikzpicture}%
    }
    
    \def\tikzsetzekreise[#1]#2#3{%
        \tikzsetzepfeil{%
        [rounded corners,#1]%
            ([shift={(-\tabcolsep,0.75\baselineskip)}]#2)%
            rectangle%
            ([shift={(\tabcolsep,-0.5\baselineskip)}]#3)
        }%
    }

    \tikzset{
        >=stealth,
        auto,
        node distance=1cm,
        thick,
        main node/.style={
            circle,draw,font=\sffamily\Large\bfseries,minimum size=0pt
        },
        state/.style={minimum size=0pt}
        loop above right/.style={loop,out=30,in=60,distance=0.5cm},
        loop above left/.style={above left,out=150,in=120,loop},
        loop below right/.style={below right,out=330,in=300,loop},
        loop below left/.style={below left,out=240,in=210,loop},
        itria/.style={
            draw,dashed,shape border uses incircle,
            isosceles triangle,shape border rotate=90,yshift=-1.45cm
        },
        rtria/.style={
            draw,dashed,shape border uses incircle,
            isosceles triangle,isosceles triangle apex angle=90,
            shape border rotate=-45,yshift=0.2cm,xshift=0.5cm
        },
        ritria/.style={
            draw,dashed,shape border uses incircle,
            isosceles triangle,isosceles triangle apex angle=110,
            shape border rotate=-55,yshift=0.1cm
        },
        litria/.style={
            draw,dashed,shape border uses incircle,
            isosceles triangle,isosceles triangle apex angle=110,
            shape border rotate=235,yshift=0.1cm
        }
    }

\lateinabkuerzung{cf}{cf.}
\lateinabkuerzung{Cf}{Cf.}
\lateinabkuerzung{idest}{i.e.}
\lateinabkuerzung{Idest}{I.e.}
\lateinabkuerzung{exempli}{e.g.}
\lateinabkuerzung{Exempli}{E.g.}
\lateinabkuerzung{etcetera}{etc.}
\lateinabkuerzung{etAlia}{et al.}
\lateinabkuerzung{viz}{viz.}
\deutscheabkuerzung{usw}{usw.}
\deutscheabkuerzung{oBdA}{o.\,B.\,d.\,A.}
\deutscheabkuerzung{OBdA}{O.\,B.\,d.\,A.}
\deutscheabkuerzung{oae}{o.\,\"A.}
\deutscheabkuerzung{oE}{o.\,E.}
\deutscheabkuerzung{OE}{O.\,E.}
\deutscheabkuerzung{og}{o.\,g.}
\deutscheabkuerzung{obengen}{o.\,g.}
\deutscheabkuerzung{obenst}{o.\,s.}
\deutscheabkuerzung{untenst}{u.\,s.}
\deutscheabkuerzung{Tfae}{T.\,f.\,a.\,e.}
\deutscheabkuerzung{tfae}{t.\,f.\,a.\,e.}
\deutscheabkuerzung{wrt}{wrt.}
\deutscheabkuerzung{withoutlog}{w.\,l.\,o.\,g.} 
\deutscheabkuerzung{Withoutlog}{W.\,l.\,o.\,g.} 
\deutscheabkuerzung{respectively}{resp.}

\def\oBall#1_#2{\cal{B}_{#2}(#1)}
\def\clBall#1_#2{\quer{\cal{B}}_{#2}(#1)}

\def\Cts{\@ifnextchar_{\Cts@tief}{\Cts@tief_{}}}
    \def\Cts@tief_#1#2{\@ifnextchar\bgroup{\Cts@two_{#1}{#2}}{\Cts@one_{#1}{#2}}}
    \def\Cts@one_#1#2{C_{#1}^{0}\big(#2\big)}
    \def\Cts@two_#1#2#3{C_{#1}^{0}\big(#2,~#3\big)}

\def\reals{\mathbb{R}}

\def\rtnl{\mathbb{Q}}
\def\intgr{\mathbb{Z}}
\def\onematrix{\text{\upshape\bfseries I}}

\def\symmdiff{\mathbin{\Delta}}

\def\UpperTrUnit{\mathop{\mathrm{UT}_{1}}}
\def\Heisenberg_#1{\mathop{\mathrm{H}_{#1}}}

\def\ntrlpos{\mathbb{N}}

\def\realsNonNeg{\reals_{+}}

\def\leer{\emptyset}

\def\BRAKET#1#2{\langle{}#1,~#2{}\rangle}

\def\dee{\textup{d}}
\def\einser{\text{\textbf{1}}}
\def\C0{\ensuremath{C_{0}}}

\def\ohne{\setminus}

\def\eps{\varepsilon}
\let\altphi\phi
\let\altvarphi\varphi
    \def\phi{\altvarphi}
    \def\varphi{\altphi}

\def\quer#1{\overline{#1}}

\def\lim{\mathop{\ell\mathrm{im}}}

\def\co{\mathop{\textit{co}}}

\def\BoundedOps#1{\@ifnextchar\bgroup{\BoundedOps@two{#1}}{\mathop{\mathfrak{L}}(#1)}}
    \def\BoundedOps@two#1#2{\mathop{\mathfrak{L}}(#1,#2)}

\def\RaumX{X}

\def\hardSigma02{\cal{Q}}

    \def\topWOT{\text{\upshape \scshape wot}}

    \def\topSOT{\text{\upshape \scshape sot}}

\@ifundefined{setcitestyle}{%
}{%
    \setcitestyle{numeric-comp,open={[},close={]}}
}

\renewcommand{\arraystretch}{1}
\def\firstparagraph{\noindent}
\def\continueparagraph{\noindent}

\hypersetup{
    hidelinks=true,
}

    \generatenestedsecnumbering{arabic}{section}{subsection}
    \generatenestedsecnumbering{arabic}{subsection}{subsubsection}
    \def\theunitnamesection{\thesection}

    \def\sectionname{}

    \let\appendix@orig\appendix
    \def\appendix{%
        \appendix@orig%
        \let\boolinappendix\boolwahr
        \addcontentsline{toc}{part}{\appendixname}%
        \addtocontents{toc}{\protect\setcounter{tocdepth}{0}}
        \def\sectionname{Appendix}%
        \def\theunitnamesection{\Alph{section}}%
    }
    \def\notappendix{%
        \let\boolinappendix\boolfalse
        \addtocontents{toc}{\protect\setcounter{tocdepth}{1 }}
        \def\sectionname{}%
        \def\theunitnamesection{\arabic{section}}%
    }

    \def\@seccntformat#1{%
        \protect\textup{%
            \protect\@secnumfont
            \expandafter\protect\csname format#1\endcsname%
            \csname the#1\endcsname
            \expandafter\protect\csname format#1@pt\endcsname%
            \space
        }%
    }

    \def\formatsection@text{\centering\Large\scshape}
    \def\formatsection@pt{\secnumberingseppt}
    \def\section{\@startsection{section}{1}{\z@}{.7\linespacing\@plus\linespacing}{.5\linespacing}{\formatsection@text}}

    \def\formatsubsection@text{\flushleft\bfseries\scshape}
    \def\formatsubsection@pt{\subsecnumberingseppt}
    \def\subsection{\@startsection{subsection}{2}{\z@}{\z@}{\z@\hspace{1em}}{\formatsubsection@text}}

\def\rafootnotectr{20}
\def\incrftnotectr#1{%
    \addtocounter{#1}{1}%
    \ifnum\value{#1}>\rafootnotectr\relax
        \setcounter{#1}{0}%
    \fi%
}
\def\footnoteref[#1]{\protected@xdef\@thefnmark{\ref{#1}}\@footnotemark}
\let\altfootnotetext\footnotetext
    \def\footnotetext[#1]#2{\incrftnotectr{footnote}\altfootnotetext[\value{footnote}]{\label{#1}#2}}

    \def\footnotemark[#1]{\text{\textsuperscript{\getrefnumber{#1}}}}

\DefineFNsymbols*{custom}{abcdefghijklmnopqrstuvwxyz}
\setfnsymbol{custom}

\def\kopfzeiledefault{
    \lhead[]{}
    \lhead[]{}
    \chead[]{}
    \rhead[]{}
    \lfoot[]{}
    \cfoot{\footnotesize\thepage}
    \rfoot[]{}
}

\def\aktuellesfont{\rmfamily}

\def\documentfont{%
    \gdef\aktuellesfont{\rmfamily}%
    \fontfamily{cmr}\fontseries{m}\selectfont%
    \renewcommand{\sfdefault}{phv}%
    \renewcommand{\ttdefault}{pcr}%
    \renewcommand{\rmdefault}{cmr}
    \renewcommand{\bfdefault}{bx}%
    \renewcommand{\itdefault}{it}%
    \renewcommand{\sldefault}{sl}%
    \renewcommand{\scdefault}{sc}%
    \renewcommand{\updefault}{n}%
}

\allowdisplaybreaks

\def\startdocumentlayoutoptions{
    \selectlanguage{british}
    \setlength{\parskip}{0.25\baselineskip}
    \setlength{\parindent}{2em}
    \kopfzeiledefault
    \documentfont
    \normalsize
}

\def\highlightTerm#1{\emph{#1}}

\def\@adminfootnotes{%
    \let\@makefnmark\relax
    \let\@thefnmark\relax
    \ifx\@empty\@date\else%
        \@footnotetext{\@setdate}%
    \fi%
    \ifx\@empty\@subjclass\else%
        \@footnotetext{\@setsubjclass}%
    \fi
    \ifx\@empty\@keywords\else%
        \@footnotetext{\@setkeywords}%
    \fi
    \ifx\@empty\thankses\else%
        \@footnotetext{\def\par{\let\par\@par}\@setthanks}%
    \fi
}

\def\@settitle{\Large\bfseries\scshape\@title}

\def\@maketitle{%
  \normalfont\normalsize
  \@adminfootnotes
  \@mkboth{\@nx\shortauthors}{\@nx\shorttitle}%
  \global\topskip42\p@\relax
  {\centering\@settitle}
  \ifx\@empty\authors\else{\centering\small\@setauthors}\fi
  \ifx\@empty\@date\else{\vtop{\centering\small\@date\@@par}}\fi
  \ifx\@empty\@dedicatory%
  \else%
    \baselineskip\p@
    \vtop{\centering{\footnotesize\itshape\@dedicatory\@@par}%
    \global\dimen@i\prevdepth}\prevdepth\dimen@i%
  \fi
  \@setabstract
  \normalsize
  \if@titlepage
    \newpage
  \else
    \dimen@34\p@\advance\dimen@-\baselineskip
  \fi
}

\def\addresseshere{%
  \bgroup
  \setlength{\parindent}{0pt}
  \enddoc@text
  \egroup
  \let\enddoc@text\relax
}

\makeatother

\begin{document}
    \startdocumentlayoutoptions

    \thispagestyle{plain}



\def\abstractname{Abstract}
\begin{abstract}
    It is well known that weakly continuous semigroups
    defined over $\realsNonNeg$ are automatically strongly continuous.
    We extend this result to more generally defined semigroups,
    including multiparameter semigroups.
\end{abstract}


\title[On the strong continuity of generalised semigroups]{On the strong continuity of generalised semigroups}
\author{Raj Dahya}
\email{raj.dahya@web.de}
\address{%
Fakult\"at f\"ur Mathematik und Informatik\newline
Universit\"at Leipzig, Augustusplatz 10, D-04109 Leipzig, Germany
}

\maketitle

    \setcounternach{section}{1}



\section[Introduction]{Introduction}
    \label{sec:intro}

\firstparagraph
By a well-known result, under certain basic conditions, semigroups over Banach spaces are automatically continuous \wrt the strong operator topology ($\topSOT$).
Engel und Nagel proved this in
    \cite[Theorem~5.8]{Engel1999}
under the assumption of continuity \wrt the weak operator topology ($\topWOT$).
In that reference and here, semigroups are ordinarily defined over $\realsNonNeg$.
Specifically, one considers operator-valued functions,

    \begin{mathe}[mc]{rcccl}
        T &: &\realsNonNeg &\to &\BoundedOps{E},\\
    \end{mathe}

\continueparagraph
where $E$ is a Banach space, and $T$ satisfies
${T(0)=\onematrix}$ and ${T(s+t)=T(s)T(t)}$ for all ${s,t\in\realsNonNeg}$.
In other words, semigroups are nothing other than morphisms
between the monoids ${(\realsNonNeg,+,0)}$ and ${(\BoundedOps{E},\circ,\onematrix)}$.

Now, for our purposes, there is no particular reason to focus on semigroups over $\realsNonNeg$,
also known as \highlightTerm{one-parameter semigroups}.
A natural abstraction is to define semigroups over topological monoids.
In this note, we shall define a broad class of semigroups,
including ones defined over $\realsNonNeg^{d}$ for $d\geq 1$,
\idest \highlightTerm{multiparameter semigroups},
and to which we generalise the automatic continuity proof in
    \cite{Engel1999}.

Our generalisation may also be of interest to other fields.
For example, multiparameter semigroups occur in
the study of diffusion equations in space-time dynamics
(see \exempli \cite{zelik2004})
and the approximation of periodic functions in multiple variables
(see \exempli \cite{terehin1975}).


\section[Definitions]{Definitions}
    \label{sec:definitions}

\firstparagraph
Note that no assumptions about commutativity shall be made,
and hence monoids and groups shall be expressed multiplicatively.
We define generalised semigroups as follows.

\begin{defn}
    A \highlightTerm{semigroup} over a Banach space, $E$, defined over a monoid, $M$,
    shall mean any operator-valued function, ${T:M\to\BoundedOps{E}}$,
    satisfying ${T(1)=\onematrix}$ and ${T(st)=T(s)T(t)}$
    for $s,t\in M$.
\end{defn}

We now define a large class of monoids
to which the classical continuity result shall be generalised.

\begin{defn}
    Let $M$ be a locally compact Hausdorff topological monoid.
    We shall call $M$ \highlightTerm{extendible},
    if there exists a locally compact Hausdorff topological group, $G$,
    such that $M$ is topologically and algebraically isomorphic to a closed subset of $G$.
\end{defn}

If $M$ is extendible to $G$ via the above definition,
then one can assume without loss of generality that $M\subseteq G$.

\begin{defn}
    Let $G$ be a locally compact Hausdorff group.
    We shall call a subset $A\subseteq G$ \highlightTerm{positive in the identity},
    if for all neighbourhoods, $U\subseteq G$, of the group identity,
    $U\cap A$ has non-empty interior within $G$.
\end{defn}

\begin{e.g.}[The non-negative reals]
\makelabel{e.g.:extendible-mon:reals:sig:article-auto-continuity-raj-dahya}
    Consider ${M \colonequals \realsNonNeg}$ viewed under addition.
    Since ${M\subseteq\reals}$ is closed,
    we have that $M$ is an extendible locally compact Hausdorff monoid.
    For any open neighbourhood, ${U\subseteq\reals}$, of the identity,
    there exists an ${\eps>0}$, such that ${(-\eps,\eps)\subseteq U}$
    and thus ${U\cap M\supseteq(0,\eps)\neq\leer}$.
    Hence $M$ is positive in the identity.
\end{e.g.}

\begin{e.g.}[The $p$-adic integers]
\makelabel{e.g.:extendible-mon:p-adic:sig:article-auto-continuity-raj-dahya}
    Consider ${M \colonequals \mathbb{Z}_{p}}$ with ${p\in\mathbb{P}}$,
    viewed under addition and with the topology generated by the $p$-adic norm.
    Since ${M\subseteq\rtnl_{p}}$ is clopen,
    we have that $M$ is an extendible locally compact Hausdorff monoid.
    Since $M$ is clopen, it is clearly positive in the identity.
\end{e.g.}

\begin{e.g.}[Discrete cases]
    Let $G$ be a discrete group, and let $M\subseteq G$ contain the identity and be closed under group multiplication.
    Clearly, $M$ is a locally compact Hausdorff monoid, extendible to $G$ and positive in the identity.
    For example one can take the free-group $\mathbb{F}_{2}$ with generators $\{a,b\}$,
    and $M$ to be the closure of $\{1,a,b\}$ under multiplication.
\end{e.g.}

\begin{e.g.}[Non-discrete, non-commutative cases]
    Let $d\in\ntrlpos$ with $d>1$ and consider the space, $\RaumX$, of $\reals$-valued $d\times d$ matrices.
    Topologised with any matrix norm (equivalently the strong or the weak operator topologies),
    this space is homeomorphic to $\reals^{d^{2}}$ and thus locally compact Hausdorff.
    Since the determinant map ${X\ni T\mapsto \det(T)\in\reals}$ is continuous,
    the subspace of invertible matrices $\{T\in\RaumX\mid \det(T)\neq 0\}$
    is open and thus a locally compact Hausdorff topological group.
    Now the subspace, $G$, of upper triangular matrices with positive diagonal entries,
    is a closed subgroup and thus locally compact Hausdorff.
    Letting

        \begin{mathe}[mc]{rcl}
            G_{0} &\colonequals &\{T\in G\mid \det(T)=1\},\\
            G_{+} &\colonequals &\{T\in G\mid \det(T)>1\},~\text{and}\\
            G_{-} &\colonequals &\{T\in G\mid \det(T)<1\},\\
        \end{mathe}

    \continueparagraph
    it is easy to see that $M \colonequals G_{0}\cup G_{+}$ is a topologically closed subspace containing the identity and is closed under multiplication.
    Moreover $M$ is a proper monoid, since the inverses of the elements in $G_{+}$ are clearly in ${G\ohne M}$.
    Consider now an open neighbourhood, ${U\subseteq G}$, of the identity.
    Since inversion is continuous, $U^{-1}$ is also an open neighbourhood of the identity.
    Since, as a locally compact Hausdorff space, $G$ satisfies the Baire category theorem,
    and since ${G_{+}\cup G_{-}}$ is clearly dense (and open) in $G$, and thus comeagre,
    we clearly have $(U\cap U^{-1})\cap(G_{+}\cup G_{-})\neq\leer$.
    So either ${U\cap G_{+}\neq\leer}$ or else ${U^{-1}\cap G_{-}\neq\leer}$,
    from which it follows that ${U\cap G_{+}=(U^{-1}\cap G_{-})^{-1}\neq\leer}$.
    Hence in each case ${U\cap M}$ contains a non-empty open subset, \viz ${U\cap G_{+}}$.
    So $M$ is extendible to $G$ and positive in the identity.

    Next, consider the subgroup, ${G_{h} \subseteq G}$,
    consisting of matrices of the form $T=\onematrix+E$ where $E$ is a strictly upper triangular matrix
    with at most non-zero entries on the top row and right hand column.
    That is, $G_{h}$ is the \highlightTerm{continuous Heisenberg group}, $\Heisenberg_{2d-3}(\reals)$, of order ${2d-3}$.
    The elements of the Heisenberg group occur in the study of Kirillov's \highlightTerm{orbit method}
        (\cf \cite{kirillov1962})
    and have important applications in physics
        (see \exempli \cite{kirillov2003}).
    Clearly, $G_{h}$ is topologically closed within $G$ and thus locally compact Hausdorff.
    Now consider the subspace,

        \begin{mathe}[mc]{rcl}
            M_{h} &\colonequals &\{T\in G_{h}\mid \forall{i,j\in\{1,2,\ldots,d\}:~}T_{ij}\geq 0\},\\
        \end{mathe}

    \continueparagraph
    of matrices with only non-negative entries.
    This is clearly a topologically closed subspace of $G_{h}$ containing the identity and closed under multiplication.
    Moreover, if ${S,T\in M_{h}\ohne\{\onematrix\}}$ we clearly have

        \begin{mathe}[mc]{rcccl}
            ST &= &\onematrix + \big((S-\onematrix) + (T-\onematrix) + (S-\onematrix)(T-\onematrix)\big)
              &\in &M_{h}\ohne\{\onematrix\},\\
        \end{mathe}

    \continueparagraph
    which implies that no non-trivial element in $M_{h}$ has its inverse in $M_{h}$, making $M_{h}$ a proper monoid.
    Consider now an open neighbourhood, ${U\subseteq G_{h}}$, of the identity.
    Since $G_{h}$ is homeomorphic to $\reals^{2d-3}$,
    there exists some ${\eps>0}$, such that

        \begin{mathe}[mc]{rcl}
            U &= &\{
                T\in G_{h}
                \mid
                \forall{(i,j)\in\mathcal{I}:~}T_{ij}\in(-\eps,\eps)
            \}.\\
        \end{mathe}

    \continueparagraph
    where $\mathcal{I} \colonequals \{(1,2),(1,3),\ldots,(1,d),(2,d),\ldots,(d-1,d)\}$.
    Hence

        \begin{mathe}[mc]{rcl}
            U\cap M_{h} &\supseteq &\{
                T\in G_{h}
                \mid
                \forall{(i,j)\in\mathcal{I}:~}T_{ij}\in(0,\eps)
            \}=:V,\\
        \end{mathe}

    \continueparagraph
    where $V$ is clearly a non-empty open subset of $G_{h}$,
    since the $2d-3$ entries in the matrices can be freely and independently chosen.
    Thus $M_{h}$ is extendible to $G_{h}$ and positive in the identity.

    Finally, we may consider the subgroup, ${G_{u} \colonequals \UpperTrUnit(d)}$, of upper triangular matrices over $\reals$ with unit diagonal.
    The elements of $\UpperTrUnit(d)$ have important applications in image analysis
    (see \exempli
        \cite{kirillov2003}
        and
        \cite[\S{}5.5.2]{pennec2020}%
    )
    and representations of the group have been studied in
        \cite[Chapter~6]{samoilenko1991}.
    Setting ${M_{u} \colonequals \{T\in G_{u}\mid \forall{i,j\in\{1,2,\ldots,d\}:~}T_{ij}\geq 0\}}$,
    one may argue similarly to the case of continuous Heisenberg group
    and show $G_{u}$ is locally compact
    and that $M_{u}$ is a proper topological monoid
    which is furthermore extendible to $G_{u}$ and positive in the identity.
\end{e.g.}

The following result allows us to generate more examples from basic ones:

\begin{prop}
\makelabel{prop:products-of-good-are-good:sig:article-auto-continuity-raj-dahya}
    Let ${n\in\ntrlpos}$ and let $M_{i}$ be locally compact Hausdorff monoids for ${1\leq i\leq n}$.
    Assume for each ${i<n}$ that $M_{i}$ is extendible to a locally compact Hausdorff group $G_{i}$,
    and that $M_{i}$ is positive in the identity of $G_{i}$.
    Then ${M \colonequals \prod_{i=1}^{n}M_{i}}$ is a locally compact Hausdorff monoid
    which is extendible to ${G \colonequals \prod_{i=1}^{n}G_{i}}$
    and positive in the identity.
\end{prop}

    \begin{proof}
        The extendibility of $M$ to $G$ is clear.
        Now consider an arbitrary open neighbourhood, $U$, of the identity in $G$.
        For each ${1\leq i\leq n}$, one can find open neighbourhoods, $U_{i}$, of the identity in $G_{i}$,
        so that ${U' \colonequals \prod_{i=1}^{n}U_{i}\subseteq U}$.
        By assumption, $M_{i}\cap U_{i}$ contains a non-empty open set,
        ${V_{i}\subseteq G_{i}}$ for each ${1\leq i\leq n}$.
        Since ${U\supseteq\prod_{i=1}^{n}U_{i}}$,
        it follows that
            $M\cap U
                \supseteq
                \prod_{i=1}^{n}(M_{i}\cap U_{i})
                \supseteq
                \prod_{i=1}^{n}V_{i}\neq\leer$.
        Thus ${M\cap U}$ has non-empty interior.
        Hence $M$ is positive in the identity.
    \end{proof}

Thus the class of monoids to which we can generalise the continuity result is infinite.
For example, by \Cref{prop:products-of-good-are-good:sig:article-auto-continuity-raj-dahya},
    \Cref{%
        e.g.:extendible-mon:reals:sig:article-auto-continuity-raj-dahya,%
        e.g.:extendible-mon:p-adic:sig:article-auto-continuity-raj-dahya%
    }
yield:

\begin{cor}
\makelabel{cor:example-multiparam-are-good:sig:article-auto-continuity-raj-dahya}
    Viewed under pointwise addition,
    $\realsNonNeg^{d}$, $\intgr_{p}^{d}$
    are extendible to locally compact Hausdorff groups
    and positive in the identities of the respective groups,
    for all ${d\in\ntrlpos}$ and ${p\in\mathbb{P}}$.
\end{cor}


\section[Main result]{Main result}
    \label{sec:main-result}

\firstparagraph
We can now state and prove the generalisation.
Our argumentation builds on
    \cite[Theorem~5.8]{Engel1999}.

\begin{schattierteboxdunn}[backgroundcolor=leer,nobreak=true]
\begin{thm}
\makelabel{thm:generalised-auto-continuity:sig:article-auto-continuity-raj-dahya}
    Let $M$ be a locally compact Hausdorff monoid and $E$ a Banach space.
    Assume that $M$ is extendible to a locally compact Hausdorff group, $G$,
    and that $M$ is positive in the identity.
    Then all $\topWOT$-continuous semigroups, ${T:M\to\BoundedOps{E}}$,
    are automatically $\topSOT$-continuous.
\end{thm}
\end{schattierteboxdunn}

    \begin{proof}
        First note that the principle of uniform boundedness applied twice to the $\topWOT$-continuous function, $T$,
        ensures that $T$ is norm-bounded on all compact subsets of $M$.
        Fix now a left-invariant Haar measure, $\lambda$, on $G$ and set

            \begin{mathe}[mc]{rcl}
                S &\colonequals &\{F\subseteq G\mid F~\text{a compact neighbourhood of the identity}\}.\\
            \end{mathe}

        Consider arbitrary ${F\in S}$ and ${x\in E}$.
        By the closure of $M$ in $G$ as well as positivity in the identity,
        ${M\cap F}$ is compact and contains a non-empty open subset of $G$.
        It follows that ${0<\lambda(M\cap F)<\infty}$.
        The $\topWOT$-continuity of $T$, the compactness (and thus measurability) of ${M\cap F}$,
        and the norm-boundedness of $T$ on compact subsets
        ensure that

            \begin{mathe}[mc]{rcl}
            \eqtag[eq:defn-xF:\beweislabel]
                \BRAKET{x_{F}}{\phi}
                    &\colonequals &\frac{1}{\lambda(M\cap F)}
                            \displaystyle\int_{t\in M\cap F}
                            \BRAKET{T(t)x}{\phi}
                            ~\dee{}t,
                    \quad\text{for $\phi\in E^{\prime}$}\\
            \end{mathe}

        \continueparagraph
        describes a well-defined element $x_{F}\in E^{\prime\prime}$.
        Exactly as in
            \cite[Theorem~5.8]{Engel1999},
        one may now argue by the $\topWOT$-continuity of $T$ and compactness of ${M\cap F}$
        that in fact ${x_{F}\in E}$ for each ${x\in E}$ and ${F\in S}$.
        Moreover, since $M$ is locally compact, one can readily see that each $x\in E$
        can be weakly approximated by the net,
            $(x_{F})_{F\in S}$,
        ordered by inverse inclusion. So

            \begin{mathe}[mc]{rcl}
                D &\colonequals &\{x_{F}\mid x\in E,~F\in S\}\\
            \end{mathe}

        \continueparagraph
        is weakly dense in $E$.
        Since the weak and strong closures of any convex subset in a Banach space coincide
        (\cf \cite[Theorem~5.98]{aliprantis2005}),
        it follows that the convex hull, $\co(D)$, is strongly dense in $E$.

        Now, to prove the $\topSOT$-continuity of $T$,
        we need to show that

            \begin{mathe}[mc]{rcl}
            \eqtag[eq:map:\beweislabel]
                t\in M &\mapsto &T(t)x\in E\\
            \end{mathe}

        \continueparagraph
        is strongly continuous for all $x\in E$.
        Since $M$ is locally compact and $T$ is norm-bounded on compact subsets of $M$,
        the set of $x\in E$ such that \eqcref{eq:map:\beweislabel} is strongly continuous,
        is itself a strongly closed convex subset of $E$.
        So, since $\co(D)$ is strongly dense in $E$,
        it suffices to prove
        the strong continuity of \eqcref{eq:map:\beweislabel}
        for each ${x\in D}$.

        To this end, fix arbitrary ${x\in E}$, ${F\in S}$ and ${t\in M}$.
        We need to show that ${T(t')x_{F}\longrightarrow T(t)x_{F}}$
        strongly for ${t'\in M}$ as ${t'\longrightarrow t}$.

        First recall, that by basic harmonic analysis,
        the canonical \highlightTerm{left-shift},

            \begin{mathe}[mc]{rcccl}
                L &: &G &\to &\BoundedOps{L^{1}(G)},\\
            \end{mathe}

        \continueparagraph
        defined via ${(L_{t}f)(s)=f(t^{-1}s)}$ for ${s,t\in G}$ and $f\in L^{1}(G)$,
        is an $\topSOT$-continuous morphism
        (\cf
            \cite[Proposition~3.5.6 ($\lambda_{1}$--$\lambda_{4}$)]{reiter2000}%
        ).
        Now, by compactness, ${f \colonequals \einser_{M\cap F}\in L^{1}(G)}$
        and it is easy to see that
            ${\|L_{t'}f-L_{t}f\|_{1}=\lambda(t'(M\cap F)\symmdiff t(M\cap F))}$
        for ${t'\in M}$.
        The $\topSOT$-continuity of $L$ thus yields

            \begin{mathe}[mc]{rcl}
            \eqtag[eq:1:\beweislabel]
                \lambda(t'(M\cap F)\symmdiff t(M\cap F)) &\longrightarrow &0\\
            \end{mathe}

        \continueparagraph
        for ${t'\in M}$ as ${t'\longrightarrow t}$.

        Fix now a compact neighbourhood, ${K\subseteq G}$, of $t$.
        For ${t'\in M\cap K}$ and ${\phi\in E^{\prime}}$ one obtains

            \begin{mathe}[mc]{rcl}
                |\BRAKET{T(t')x_{F}-T(t)x_{F}}{\phi}|
                    &= &|\BRAKET{x_{F}}{T(t')^{\ast}\phi}-\BRAKET{x_{F}}{T(t)^{\ast}\phi}|\\
                    &= &\frac{1}{\lambda(M\cap F)}
                        \cdot\left|
                            \displaystyle\int_{s\in M\cap F}\textstyle
                            \BRAKET{T(s)x}{T(t')^{\ast}\phi}~\dee{}s
                        -
                            \displaystyle\int_{s\in M\cap F}\textstyle
                            \BRAKET{T(s)x}{T(t)^{\ast}\phi}~\dee{}s
                        \right|\\
                    &\multispan{2}{\text{by construction of $x_{F}$ in \eqcref{eq:defn-xF:\beweislabel}}}\\
                    &= &\frac{1}{\lambda(M\cap F)}
                        \cdot\left|
                            \displaystyle\int_{s\in M\cap F}\textstyle
                            \BRAKET{T(t's)x}{\phi}~\dee{}s
                        -
                            \displaystyle\int_{s\in M\cap F}\textstyle
                            \BRAKET{T(ts)x}{\phi}~\dee{}s
                        \right|\\
                    &\multispan{2}{\text{since $T$ is a semigroup}}\\
                    &= &\frac{1}{\lambda(M\cap F)}
                        \cdot\left|
                            \displaystyle\int_{s\in t'(M\cap F)}\textstyle
                            \BRAKET{T(s)x}{\phi}~\dee{}s
                        -
                            \displaystyle\int_{s\in t(M\cap F)}\textstyle
                            \BRAKET{T(s)x}{\phi}~\dee{}s
                        \right|\\
                    &\multispan{2}{\text{by left-invariance}}\\
                    &\leq &\frac{1}{\lambda(M\cap F)}
                            \displaystyle\int_{s\in t'(M\cap F)\symmdiff t(M\cap F)}\textstyle
                            |\BRAKET{T(s)x}{\phi}|~\dee{}s\\
                    &\leq &\frac{1}{\lambda(M\cap F)}
                        \cdot\displaystyle\sup_{s\in(M\cap K)(M\cap F)}\textstyle\|T(s)\|
                        \cdot\|x\|\cdot\|\phi\|
                        \cdot\lambda(t'(M\cap F)\symmdiff t(M\cap F))\\
                        &\multispan{2}{\text{since $t,t'\in M\cap K$}.}\\
            \end{mathe}

        Since $K' \colonequals (M\cap K)(M\cap F)$ is compact, and $T$ is uniformly bounded on compact sets (see above),
        it holds that ${C \colonequals \displaystyle\sup_{s\in K'}\|T(s)\|\textstyle<\infty}$.
        The above calculation thus yields

            \begin{mathe}[mc]{rcl}
            \eqtag[eq:2:\beweislabel]
                \|T(t')x_{F}-T(t)x_{F}\|
                &= &\displaystyle\sup\textstyle\{
                        |\BRAKET{T(s)x_{F}-T(t)x_{F}}{\phi}|
                    \mid
                        \phi\in E^{\prime},~\|\phi\|\leq 1
                    \}\\
                &\leq &\frac{1}{\lambda(M\cap F)}
                        \cdot C
                        \cdot\|x\|
                        \cdot\lambda(t'(M\cap F)\symmdiff t(M\cap F))\\
            \end{mathe}

        \continueparagraph
        for all $t'\in M$ sufficiently close to $t$.

        By \eqcref{eq:1:\beweislabel}, the right-hand side of \eqcref{eq:2:\beweislabel} converges to $0$
        and hence ${T(t')x_{F}\longrightarrow T(t)x_{F}}$
        strongly as ${t'\longrightarrow t}$.
        This completes the proof.
    \end{proof}

\Cref{thm:generalised-auto-continuity:sig:article-auto-continuity-raj-dahya}
applied to
\Cref{cor:example-multiparam-are-good:sig:article-auto-continuity-raj-dahya}
immediately yields:

\begin{cor}
    Let $d\in\ntrlpos$ and let $E$ be a Banach space.
    Then all $\topWOT$-continuous semigroups, ${T:\realsNonNeg^{d}\to\BoundedOps{E}}$,
    are automatically $\topSOT$-continuous.
\end{cor}

\begin{rem}
    In the proof of \Cref{thm:generalised-auto-continuity:sig:article-auto-continuity-raj-dahya},
    weak continuity only played a role in obtaining
    the boundedness of $T$ on compact sets,
    as well as the well-definedness of the elements in $D$.
    Now, another proof of the classical result exists under weaker conditions,
    \viz weak measurability, provided the semigroups are almost separably valued
    (\cf
    \cite[Theorem~9.3.1 and Theorem~10.2.1--3]{hillephillips2008}%
    ).
    It remains open, whether the approach in
        \cite{hillephillips2008}
    can be adapted to our context,
    to yield the result under weaker assumptions.
\end{rem}




\subsection*{Acknowledgement}

\noindent
The author is grateful to Konrad Zimmermann for his helpful comments on a preliminary version.

\bibliographystyle{acm}
\def\bibname{References}
\bgroup
\footnotesize


\egroup

    \addresseshere
\end{document}